\documentclass{birkjour}
\usepackage{cite}
\usepackage{amsmath}
\usepackage{amssymb}
\allowdisplaybreaks
\usepackage{amsthm}
\usepackage{etoolbox}
\usepackage{xcolor}
\usepackage{graphicx}
\definecolor{navyblue}{RGB}{0,0,128}
\apptocmd{\lim}{\limits}{}{}
\newtheorem{thm}{Theorem}[section]

\newtheorem{lem}[thm]{Lemma}

\theoremstyle{definition}

\theoremstyle{remark}

\numberwithin{equation}{section}
\newcommand{\xdownarrow}[1]{
	{\left\downarrow\vbox to #1{}\right.\kern-\nulldelimiterspace}}

\usepackage{xcolor}
\usepackage{hyperref}
\hypersetup{
	colorlinks=true,
	linkcolor=blue,
	urlcolor=navyblue,
	citecolor=blue
}

\begin{document}
		\title[]{Inverses of six classes of permutation polynomials of the form $x+\gamma\operatorname{Tr}_q^{q^2}(h(x))$ over finite fields of even characteristic} 
	\author[Rajesh P. Singh]{Rajesh P. Singh}
	\address{Department of Mathematics\\ Central University of South Bihar, Gaya, India}
	\email{rpsingh@cub.ac.in}
	
	\author[Dinesh Kumar]{Dinesh Kumar}
	\address{Department of Mathematics\\ Central University of South Bihar, Gaya, India}
	\email{kumardinesh63587@gmail.com}
	\author[Jitendra Prakash]{Jitendra Prakash}
	\address{Department of Mathematics\\ Central University of South Bihar, Gaya, India}
	\email{jitendraprakash96@gmail.com }
	

	\begin{abstract}
	Recently, Jiang et al. \cite{JIANG2025102522} obtained several classes of Permutation Polynomials of the form $x+\gamma\operatorname{Tr}_q^{q^2}(h(x))$ over finite fields $\mathbb{F}_{q^2},q=2^n$. In this paper, we find the compositional inverses of six classes of permutation polynomials of this form. 
	\end{abstract}

	\maketitle

	\section{Introduction}
	Let \( q \) be a prime power and $\mathbb{F}_q$ be a finite field with $q$ elements. We call a polynomial  $f(x) \in \mathbb{F}_q[x]$ a \textit{permutation polynomial} (PP) of \( \mathbb{F}_q \) if the map \( f : a \mapsto f(a) \) from $\mathbb{F}_q$ to $\mathbb{F}_q$ is a bijection. Permutation polynomials garnered significant attention for over a century, dating back to the foundational work of Hermite and Dickson. There has been a significant progress in finding new classes of permutation polynomials over finite fields in the last two decades.

	It is well known that every mapping from $\mathbb{F}_{q}$ to $\mathbb{F}_{q}$ can be represented by a unique polynomial over $\mathbb{F}_{q}$ of degree at most
	$q-1$. So the inverse map of a permutation polynomial is given by a  polynomial over $\mathbb{F}_{q}$ of degree at most $q-1$. The unique polynomial $g(x)$ such that $f(g(x))=g(f(x))=x$ for each $x \in \mathbb{F}_{q}$ is called the \textit{compositional inverse} of $f(x)$.

	In 1991, Mullen \cite{Mullen G.L.} posed the problem of efficiently computing the coefficients of the compositional inverse of a permutation polynomial. While the existence of such an inverse is guaranteed by the bijective nature of $f(x)$, determining the explicit algebraic form of $g(x)$ is a difficult problem.  The Lagrange interpolation formula provides a theoretical solution but typically yields a polynomial with all possible terms, offering little insight into the algebraic structure and being highly computationally complex.  Explicitly obtaining the compositional inverse of a permutation polynomial is useful, as many applications require both permutation polynomials and their inverses. 
	For example, in Block ciphers like the Advanced Encryption Standard (AES), PPs and their inverses plays an essential role. The PPs and their inverses are also useful in public key cryptography  and coding theory \cite{rajesh1, rajesh2, ding2013cyclic, li2023further, hou1997results}.
	
	In recent years, the compositional inverses of several classes of permutation polynomials have been obtained. Q. Wang \cite{ wang2025survey} has surveyed the existing findings on compositional inverses of permutation polynomials over finite fields and summarized the main techniques used to study them. These techniques include the piecewise method, decomposition method, commutative diagram method, local method, and the experimental method.
	We refer the interested readers to 
	  \cite{doi:10.1142/S0219498822502206,  coulter2002compositional, li2019compositional, muratovic2007note, niu2021finding, tuxanidy2014inverses, tuxanidy2017compositional, wang2009inverse, wang2017note, wu2014compositional, wu2024permutation, yuan2022compositional, yuan2024local, zheng2018inverses, zheng2019inverses, zheng2023inverses,wu2025compositional2,wu2025compositional} for more details.

 In this paper, we find the compositional inverses of some permutation polynomials of the form $x+\gamma\operatorname{Tr}_q^{q^2}(h(x))$ by Jiang et al. \cite{JIANG2025102522}. In particular, we find the compositional inverses of the following classes of permutation polynomials. 
	\begin{thm}\cite{JIANG2025102522}
	Let $q = 2^m$ and 
	\(
	f_1(x) \;=\; x + \gamma \, \operatorname{Tr}_{q}^{q^2}(x^3 + x^{q+2}).
	\)
	Then $f_1(x)$ is a permutation polynomial of $\mathbb{F}_{q^2}$ if and only if $\gamma \in \mathbb{F}_q$.
\end{thm}

\begin{thm}\cite{JIANG2025102522}
	Let $q = 2^m$, 
	\(
	f_2(x) = x + \gamma \, \operatorname{Tr}_{q}^{q^2}(x + x^2 + x^3 + x^{q+2}).
	\)
	Then $f_2(x)$ permutes $\mathbb{F}_{q^2}$ if and only if $\gamma \in \mathbb{F}_q$.
\end{thm}

\begin{thm}\cite{JIANG2025102522}
	Let $q = 2^m$, 
	\(
	f_3(x) = x + \gamma \, \operatorname{Tr}_{q}^{q^2}(x + x^3 + x^{q+2}).
	\)
	Then $f_3(x)$ is a permutation polynomial of $\mathbb{F}_{q^2}$ if and only if one of the following holds:
	\begin{enumerate}
		\item[(a)] $\gamma \in \mathbb{F}_q$, if $m$ is even;
		\item[(b)] $\operatorname{Tr}_{q}^{q^2}(\gamma + \gamma^2) = 0$, if $m$ is odd.
	\end{enumerate}
\end{thm}

\begin{thm}\cite{JIANG2025102522}
	Let $q = 2^m$, 
	\(
	f_4(x) = x + \gamma \, \operatorname{Tr}_{q}^{q^2}(x^2 + x^3 + x^{q+2}).
	\)
	Then $f_4(x)$ permutes $\mathbb{F}_{q^2}$ if and only if
	\begin{enumerate}
		\item[(a)] $\gamma \in \mathbb{F}_q$, if $m$ is even;
		\item[(b)] $\operatorname{Tr}_{q}^{q^2}(\gamma + \gamma^2) = 0$, if $m$ is odd.
	\end{enumerate}
\end{thm}

\begin{thm}\cite{JIANG2025102522}
	Let $m$ be an odd integer and $q = 2^m$. Then $f_5(x) = x + \gamma\operatorname{Tr}_q^{q^2}(x + x^{q+2})$
	permutes $\mathbb{F}_{q^2}$ if and only if $\gamma \in \mathbb{F}_{2^2} \setminus\{1\}$.
\end{thm}
\begin{thm}\cite{JIANG2025102522}
	Let $m$ be an odd integer and $q = 2^m$. Then $f_6(x) = x + \gamma \operatorname{Tr}_q^{q^2}(x^2 + x^{q+2})$
	permutes $\mathbb{F}_{q^2}$ if and only if $\gamma \in \mathbb{F}_{2^2} \setminus\{1\}$.
\end{thm}

In Section 2, we present the auxiliary results required for the paper. In section 3, we present the compositional inverses of the above listed permutation polynomials.
	
	\section{Auxiliary results}
Throughout the paper we take $q=2^m$. Let ${\mbox{Tr}}_{q}^{q^2}(\alpha x)=\alpha x+\alpha^{q}x^{q}\in \mathbb{F}_{q^{2}}[x]$ be a linearized polynomial. We can easily see that the mapping from $(\mathbb{F}_{q^{2}}, +)$ to $(\mathbb{F}_{q}, +)$ given by ${\mbox{Tr}}_{q}^{q^2}(\alpha x)$ is a surjective homomorphism. Clearly, $\Large({\mbox{Tr}}_{q}^{q^2}(\alpha x)\Large)^{q}={\mbox{Tr}}_{q}^{q^2}(\alpha x)$ for each $x \in \mathbb{F}_{q^{2}}$.

Next two Lemmas are used in the sequel.
	\begin{lem} \label{lem0}
	Let $\mathbb{F}_{2^{m}}$ be a finite field with $m$ an odd number. Then	$f(x)=x^3$ is a permutation polynomial over $\mathbb{F}_{2^{m}}$, and its inverse $f^{-1}(x)=x^{t}$ where $3t\equiv 1\mod{(2^m-1)}$.
	\end{lem}
	\begin{proof}
		The proof is trivial by noting the fact that $\gcd(3,2^m-1)=1$ for an odd integer $m$. 
	\end{proof}

	\begin{lem}\label{lem1}
		 Let $q=2^m$ and $\beta,\delta\in\mathbb{F}_{q^2}$ be such that $\beta\neq c\delta$ for any $c\in \mathbb{F}^{*}_q$. Then $\beta\delta^{q}+\delta\beta^{q}\neq0$.
	\end{lem}
	\begin{proof}
		If $\beta\delta^{q}+\delta\beta^{q}=0$, then we have
		$\beta\delta^{-1}={(\beta\delta^{-1})}^{q}$. Equivalently, we have $\beta\delta^{-1}=c\in\mathbb{F}_q$ or $\beta=c\delta$, a contradiction.
	\end{proof}
The following Lemma is crucial in finding the compositional inverses of permutation polynomials in the next Section.
	\begin{lem}\label{lem2}
		Let $q=2^m$ and $\alpha\in\mathbb{F}_{q^2}\setminus\mathbb{F}_q$ with $\operatorname{Tr}_q^{q^2}(\alpha)=1$.  Let $\beta, \delta \in \mathbb{F}^{*}_{q^2}$ be such that $\beta \neq c\delta$ for any $c \in \mathbb{F}^{*}_{q}$. 
		Then 
		$T(x)=\operatorname{Tr}_{q}^{q^2}(\beta x)+\alpha\operatorname{Tr}_{q}^{q^2}(\delta x)$ is a permutation polynomial over $\mathbb{F}_{q^2}$ and its inverse is
		$$T^{-1}(x)=(\beta\delta^{q}+\delta\beta^{q})^{-1}\{\delta^{q}\operatorname{Tr}_{q}^{q^2}((1+\alpha)x)+\beta^{q}\operatorname{Tr}_{q}^{q^2}(x)\}.$$
	\end{lem}
	\begin{proof}
		We can easily see that $\operatorname{Tr}_{q}^{q^2}((1+\alpha)y)+\alpha\operatorname{Tr}_{q}^{q^2}(y)=y$. We consider the equation $T(x)=y$ or equivalently, we have 
		\begin{align}\label{eqn1}
			\operatorname{Tr}_{q}^{q^2}(\beta x)+\alpha\operatorname{Tr}_{q}^{q^2}(\delta x)=\operatorname{Tr}_{q}^{q^2}((1+\alpha)y)+\alpha\operatorname{Tr}_{q}^{q^2}(y).
		\end{align}
	Since $\{1, \alpha\}$ is a basis of $\mathbb{F}_{q^2}$ over $\mathbb{F}_{q}$, therefore from equation \ref{eqn1}, we get the following system of equations
	\begin{align*}
		\operatorname{Tr}_{q}^{q^2}(\beta x)=&\operatorname{Tr}_{q}^{q^2}((1+\alpha)y)\\
		\operatorname{Tr}_{q}^{q^2}(\delta x)=&\operatorname{Tr}_{q}^{q^2}(y).
	\end{align*}
The above system of equations can be written in matrix form as follows
	$$\begin{bmatrix}
		\beta&\beta^{q}\\
		\delta&\delta^{q}
	\end{bmatrix}
	\begin{bmatrix}
		x\\x^{q}
	\end{bmatrix}
	=\begin{bmatrix}
		\operatorname{Tr}_{q}^{q^2}((1+\alpha)y)\\\operatorname{Tr}_{q}^{q^2}(y)
	\end{bmatrix}.$$
	Applying elementary row operations, $R_1\rightarrow \delta^{q}R_1$ and $R_2\rightarrow \beta^{q}R_2$, we get\\
	\[\begin{bmatrix}
		\beta\delta^{q}&\delta^{q}\beta^{q}\\
		\delta\beta^{q}&\beta^{q}\delta^{q}
	\end{bmatrix}
	\begin{bmatrix}	
			x\\x^{q}
	\end{bmatrix}
	=\begin{bmatrix}
		\delta^{q}\operatorname{Tr}_{q}^{q^2}((1+\alpha)y)\\\beta^{q}\operatorname{Tr}_{q}^{q^2}(y)
	\end{bmatrix}.\]
	Applying elementary row operation, $R_1\rightarrow R_1+R_2$, we get\\
	\[\begin{bmatrix}
		\beta\delta^{q}+\delta\beta^{q}&0\\
		\delta\beta^{q}&\beta^{q}\delta^{q}
	\end{bmatrix}
	\begin{bmatrix}
		x\\x^{q}
	\end{bmatrix}
	=\begin{bmatrix}
		\delta^{q}\operatorname{Tr}_{q}^{q^2}((1+\alpha)y)+\beta^{q}\operatorname{Tr}_{q}^{q^2}(y)\\\beta^{q}\operatorname{Tr}_{q}^{q^2}(y)
	\end{bmatrix}.\]
	From the above system of equations, we get
	\[(\beta\delta^{q}+\delta\beta^{q})x=\delta^{q}\operatorname{Tr}_{q}^{q^2}((1+\alpha)y)+\beta^{q}\operatorname{Tr}_{q}^{q^2}(y).\]
	This implies\\
	$T^{-1}(x)=(\beta\delta^{q}+\delta\beta^{q})^{-1}\{\delta^{q}\operatorname{Tr}_{q}^{q^2}((1+\alpha)x)+\beta^{q}\operatorname{Tr}_{q}^{q^2}(x)\}$.
	\end{proof}
The following Lemma is used in the next Section.

	\begin{lem} \label{lem3}
		Let $q=2^m$, $\alpha\in\mathbb{F}_{q^2}\setminus\mathbb{F}_q$ be such that $\operatorname{Tr}_q^{q^2}(\alpha)=1$ . Then, we have 
		$\operatorname{Tr}_{q}^{q^2}(\beta T^{-1}(x))=\operatorname{Tr}_{q}^{q^2}((1+\alpha)x)$ and 
		$\operatorname{Tr}_{q}^{q^2}(\delta T^{-1}(x))=\operatorname{Tr}_{q}^{q^2}(x).$
	\end{lem}
\begin{proof}	
	We have 
		\begin{align*}
		\operatorname{Tr}_{q}^{q^2}(\beta T^{-1}(x))=&\operatorname{Tr}_{q}^{q^2}[\beta(\beta\delta^{q}+\delta\beta^{q})^{-1}\{\delta^{q}\operatorname{Tr}_{q}^{q^2}((1+\alpha)x)\\&+\beta^{q}\operatorname{Tr}_{q}^{q^2}(x)\}]\\
		=&\operatorname{Tr}_{q}^{q^2}[(\beta\delta^{q}+\delta\beta^{q})^{-1}\beta\delta^{q}\operatorname{Tr}_{q}^{q^2}((1+\alpha)x)\\&+(\beta\delta^{q}+\delta\beta^{q})^{-1}\beta\beta^{q}\operatorname{Tr}_{q}^{q^2}(x)]\\
		=&(\beta\delta^{q}+\delta\beta^{q})^{-1}\beta\delta^{q}\operatorname{Tr}_{q}^{q^2}((1+\alpha)x)\\&+(\beta\delta^{q}+\delta\beta^{q})^{-1}\beta^{q}\delta\operatorname{Tr}_{q}^{q^2}((1+\alpha)x)\\&+(\beta\delta^{q}+\delta\beta^{q})^{-1}\beta\beta^{q}\operatorname{Tr}_{q}^{q^2}(x)\\&+(\beta\delta^{q}+\delta\beta^{q})^{-1}\beta\beta^{q}\operatorname{Tr}_{q}^{q^2}(x)\\
		=&(\beta\delta^{q}+\delta\beta^{q})^{-1}\beta\delta^{q}\operatorname{Tr}_{q}^{q^2}((1+\alpha)x)\\&+(\beta\delta^{q}+\delta\beta^{q})^{-1}\beta^{q}\delta\operatorname{Tr}_{q}^{q^2}((1+\alpha)x)\\
		=&(\beta\delta^{q}+\delta\beta^{q})^{-1}(\beta\delta^{q}+\delta\beta^{q})\operatorname{Tr}_{q}^{q^2}((1+\alpha)x)\\
		=&\operatorname{Tr}_{q}^{q^2}((1+\alpha)x).
	\end{align*}
	\begin{align*}
		\operatorname{Tr}_{q}^{q^2}(\delta T^{-1}(x))=&\operatorname{Tr}_{q}^{q^2}[\delta(\beta\delta^{q}+\delta\beta^{q})^{-1}\{\delta^{q}\operatorname{Tr}_{q}^{q^2}((1+\alpha)x)\\&+\beta^{q}\operatorname{Tr}_{q}^{q^2}(x)\}]\\
		=&\operatorname{Tr}_{q}^{q^2}[(\beta\delta^{q}+\delta\beta^{q})^{-1}\delta\delta^{q}\operatorname{Tr}_{q}^{q^2}((1+\alpha)x)\\&+(\beta\delta^{q}+\delta\beta^{q})^{-1}\delta\beta^{q}\operatorname{Tr}_{q}^{q^2}(x)]\\
		=&(\beta\delta^{q}+\delta\beta^{q})^{-1}\delta\delta^{q}\operatorname{Tr}_{q}^{q^2}((1+\alpha)x)\\&+(\beta\delta^{q}+\delta\beta^{q})^{-1}\delta^{q}\delta\operatorname{Tr}_{q}^{q^2}((1+\alpha)x)\\&+(\beta\delta^{q}+\delta\beta^{q})^{-1}\delta\beta^{q}\operatorname{Tr}_{q}^{q^2}(x)\\&+(\beta\delta^{q}+\delta\beta^{q})^{-1}\beta\delta^{q}\operatorname{Tr}_{q}^{q^2}(x)\\
		=&\operatorname{Tr}_{q}^{q^2}(x).
	\end{align*}
\end{proof}
	\section{Compositional Inverses of Permutation Polynomials} 
	In this section we present the compositional inverses of six classes of permutation polynomials of the form $x+\gamma\operatorname{Tr}_q^{q^2}(h(x))$.
	\begin{thm} Let $q=2^m$, $\gamma\in \mathbb{F}_q$. Then the compositional inverse of permutation polynomial 
		$f_1(x)=x+\gamma\operatorname{Tr}_{q}^{q^2}(x^3+x^{q+2})$  over $\mathbb{F}_{q^2}$ is 
		$x+\gamma(\operatorname{Tr}_{q}^{q^2}(x))^3$.
	\end{thm}
	\begin{proof} We take $T(x)= \operatorname{Tr}_q^{q^2}(\beta x)+\alpha\operatorname{Tr}_q^{q^2}(\delta x)$ as defined in Lemma \ref{lem2} and find $f_1\circ T(x)=f_1(T(x))$.
		We have,
		\begin{align}
		f_1\circ T(x)=&\operatorname{Tr}_{q}^{q^2}(\beta x)+\alpha\operatorname{Tr}_{q}^{q^2}(\delta x)+\gamma\operatorname{Tr}_{q}^{q^2}[\{\operatorname{Tr}_{q}^{q^2}(\beta x)+\alpha\operatorname{Tr}_{q}^{q^2}(\delta x)\}^3\notag\\&+\{\operatorname{Tr}_{q}^{q^2}(\beta x)+\alpha\operatorname{Tr}_{q}^{q^2}(\delta x)\}^{q+2}]\notag\\
		=&\operatorname{Tr}_{q}^{q^2}(\beta x)+\alpha\operatorname{Tr}_{q}^{q^2}(\delta x)+\gamma\operatorname{Tr}_{q}^{q^2}[\{\operatorname{Tr}_{q}^{q^2}(\beta x)+\alpha\operatorname{Tr}_{q}^{q^2}(\delta x)\}^2\{\operatorname{Tr}_{q}^{q^2}(\beta x)\notag\\&+\alpha\operatorname{Tr}_{q}^{q^2}(\delta x)+(\operatorname{Tr}_{q}^{q^2}(\beta x)+\alpha\operatorname{Tr}_{q}^{q^2}(\delta x))^q\}]\notag\\
		=&\operatorname{Tr}_{q}^{q^2}(\beta x)+\alpha\operatorname{Tr}_{q}^{q^2}(\delta x)+\gamma\operatorname{Tr}_{q}^{q^2}[\{\operatorname{Tr}_{q}^{q^2}(\beta x)+\alpha\operatorname{Tr}_{q}^{q^2}(\delta x)\}^2\{\operatorname{Tr}_{q}^{q^2}(\beta x)\notag\\&+\alpha\operatorname{Tr}_{q}^{q^2}(\delta x)+\operatorname{Tr}_{q}^{q^2}(\beta x)+\alpha\operatorname{Tr}_{q}^{q^2}(\delta x)+\operatorname{Tr}_{q}^{q^2}(\delta x)\}]\notag\\
		=&\operatorname{Tr}_{q}^{q^2}(\beta x)+\alpha\operatorname{Tr}_{q}^{q^2}(\delta x)+\gamma[\operatorname{Tr}_{q}^{q^2}(\delta x)]^3\label{a}.
	\end{align}
For any arbitrary element $y\in \mathbb{F}_{q^2}$, $T(y)=\operatorname{Tr}_{q}^{q^2}(\beta y)+\alpha\operatorname{Tr}_{q}^{q^2}(\delta y)$ is an arbitrary element of $\mathbb{F}_{q^{2}}$.
	We consider the equation $f_1\circ T(x)=T(y)=\operatorname{Tr}_{q}^{q^2}(\beta y)+\alpha\operatorname{Tr}_{q}^{q^2}(\delta y)$. Since $\{1, \alpha\}$ is basis of $\mathbb{F}_{q^2}$ over $\mathbb{F}_{q}$,
	 comparing each side, we get the following system of equations
	\begin{align}
		\operatorname{Tr}_{q}^{q^2}(\delta x)=&\operatorname{Tr}_{q}^{q^2}(\delta y)\\
		\operatorname{Tr}_{q}^{q^2}(\beta x)=&\operatorname{Tr}_{q}^{q^2}(\beta y)+\gamma(\operatorname{Tr}_{q}^{q^2}(\delta y))^3.
	\end{align}
	We write the above equations in matrix form as follows\\
	$$\begin{bmatrix}
		\beta&\beta^{q}\\
		\delta&\delta^{q}
	\end{bmatrix}
	\begin{bmatrix}
		x\\x^{q}
	\end{bmatrix}=
	\begin{bmatrix}
		\operatorname{Tr}_{q}^{q^2}(\beta y)+\gamma(\operatorname{Tr}_{q}^{q^2}(\delta y))^3\\
		\operatorname{Tr}_{q}^{q^2}(\delta y)
	\end{bmatrix}$$.
	Applying elementary row operations $R_1\rightarrow\delta^{q}R_1$ and $R_2\rightarrow\beta^{q}R_2$, we get
	
	$$\begin{bmatrix}
		\beta\delta^{q}&\beta^{q}\delta^{q}\\
		\beta^{q}\delta&\beta^{q}\delta^{q}
	\end{bmatrix}
	\begin{bmatrix}
		x\\x^{q}
	\end{bmatrix}=
	\begin{bmatrix}
		\delta^{q}\operatorname{Tr}_{q}^{q^2}(\beta y)+\delta^{q}\gamma(\operatorname{Tr}_{q}^{q^2}(\delta y))^3\\
		\beta^{q}\operatorname{Tr}_{q}^{q^2}(\delta y)
	\end{bmatrix}$$.
	Again applying elementary row operation $R_1\rightarrow R_1+R_2$, we get
	\begin{align}
	\begin{bmatrix}
		\beta\delta^{q}+\beta^{q}\delta&0\\
		\beta^{q}\delta&\beta^{q}\delta^{q}
	\end{bmatrix}&
	\begin{bmatrix}
		x\\x^{q}
	\end{bmatrix}=A\label{b}\end{align}
where,\begin{align*} A=
	\begin{bmatrix}
		\delta^{q}\operatorname{Tr}_{q}^{q^2}(\beta y)+\delta^{q}\gamma(\operatorname{Tr}_{q}^{q^2}(\delta y))^3+\beta^{q}\operatorname{Tr}_{q}^{q^2}(\delta y)\\
		\beta^{q}\operatorname{Tr}_{q}^{q^2}(\delta y)
	\end{bmatrix}.
\end{align*}
	From the equation \ref{b}, we get\\
	$x=(\beta\delta^{q}+\beta^{q}\delta)^{-1}\{\delta^{q}\operatorname{Tr}_{q}^{q^2}(\beta y)+\delta^{q}\gamma(\operatorname{Tr}_{q}^{q^2}(\delta y))^3+\beta^{q}\operatorname{Tr}_{q}^{q^2}(\delta y)\}$\\
	or equivalently, we have
	\begin{align*}
	T^{-1}\circ f_1^{-1}\circ T(y)=&(\beta\delta^{q}+\beta^{q}\delta)^{-1}\{\delta^{q}\operatorname{Tr}_{q}^{q^2}(\beta y)+\delta^{q}\gamma(\operatorname{Tr}_{q}^{q^2}(\delta y))^3\\&+\beta^{q}\operatorname{Tr}_{q}^{q^2}(\delta y)\}.
\end{align*}
	Substituting $y=T^{-1}(x)$ in above equation, we get
	\begin{align*}
	T^{-1}\circ f_1^{-1}(x)=&(\beta\delta^{q}+\beta^{q}\delta)^{-1}\{\delta^{q}\operatorname{Tr}_{q}^{q^2}(\beta T^{-1}(x))+\delta^{q}\gamma(\operatorname{Tr}_{q}^{q^2}(\delta T^{-1}(x)))^3\\&+\beta^{q}\operatorname{Tr}_{q}^{q^2}(\delta T^{-1}(x))\}.
   \end{align*}
   Now using lemma \ref{lem3}, we get
   \begin{align}
   	T^{-1}\circ f_1^{-1}(x)=&(\beta\delta^{q}+\beta^{q}\delta)^{-1}\{\delta^{q}\operatorname{Tr}_{q}^{q^2}((1+\alpha)x)+\delta^{q}\gamma(\operatorname{Tr}_{q}^{q^2}(x))^3
   	\notag\\&+\beta^{q}\operatorname{Tr}_{q}^{q^2}(x)\}.
   \end{align}
   Since $T(x)= \operatorname{Tr}_q^{q^2}(\beta x)+\alpha\operatorname{Tr}_q^{q^2}(\delta x)$, we see that
   \begin{align}
   	f_1^{-1}(x)=&\operatorname{Tr}_{q}^{q^2}\{\beta (T^{-1}\circ f_1^{-1}(x))\}+\alpha\operatorname{Tr}_{q}^{q^2}\{\delta (T^{-1}\circ f_1^{-1}(x))\}.\label{c}
   \end{align}
   Next, we have
   \begin{align}
   	\operatorname{Tr}_{q}^{q^2}\{\beta (T^{-1}\circ f_1^{-1}(x))\}=&\operatorname{Tr}_{q}^{q^2}[\beta (\beta\delta^{q}+\beta^{q}\delta)^{-1}\{\delta^{q}\operatorname{Tr}_{q}^{q^2}((1+\alpha)x)\notag\\&+\delta^{q}\gamma(\operatorname{Tr}_{q}^{q^2}(x))^3+\beta^{q}\operatorname{Tr}_{q}^{q^2}(x)\}]\notag\\
   	=&\operatorname{Tr}_{q}^{q^2}[\beta\delta^{q}(\beta\delta^{q}+\beta^{q}\delta)^{-1}\operatorname{Tr}_{q}^{q^2}((1+\alpha)x)\notag\\&+\beta\delta^{q}(\beta\delta^{q}+\beta^{q}\delta)^{-1}\gamma(\operatorname{Tr}_{q}^{q^2}(x))^3\notag\\&+\beta\beta^{q}(\beta\delta^{q}+\beta^{q}\delta)^{-1}\operatorname{Tr}_{q}^{q^2}(x)]\notag\\
   	=&\operatorname{Tr}_{q}^{q^2}((1+\alpha)x)+\gamma(\operatorname{Tr}_{q}^{q^2}(x))^3\label{d}.
   \end{align}
   and
   \begin{align}
   	\operatorname{Tr}_{q}^{q^2}(\delta (T^{-1}\circ f_1^{-1}(x)))=&\operatorname{Tr}_{q}^{q^2}[\delta (\beta\delta^{q}+\beta^{q}\delta)^{-1}\{\delta^{q}\operatorname{Tr}_{q}^{q^2}((1+\alpha)x)\notag\\&+\delta^{q}\gamma(\operatorname{Tr}_{q}^{q^2}(x))^3+\beta^{q}\operatorname{Tr}_{q}^{q^2}(x)\}]\notag\\
   	=&\operatorname{Tr}_{q}^{q^2}[\delta\delta^{q}(\beta\delta^{q}+\beta^{q}\delta)^{-1}\operatorname{Tr}_{q}^{q^2}((1+\alpha)x)\notag\\&+\delta\delta^{q}(\beta\delta^{q}+\beta^{q}\delta)^{-1}\gamma(\operatorname{Tr}_{q}^{q^2}(x))^3\notag\\&+\delta\beta^{q}(\beta\delta^{q}+\beta^{q}\delta)^{-1}\operatorname{Tr}_{q}^{q^2}(x)]\notag\\
   	=&\operatorname{Tr}_{q}^{q^2}(x).\label{e}
   \end{align}
  Now, using equations \ref{d} and \ref{e} in equation \ref{c}, we get
   \begin{align*}
   f_1^{-1}(x)=&\operatorname{Tr}_{q}^{q^2}((1+\alpha)x)+\gamma(\operatorname{Tr}_{q}^{q^2}(x))^3+\alpha\operatorname{Tr}_{q}^{q^2}(x)\\
   =&x+\gamma(\operatorname{Tr}_{q}^{q^2}(x))^3.
\end{align*}
	\end{proof}
	\begin{thm}Let $q=2^m$, $\gamma\in \mathbb{F}_q$. Then the compositional inverse of permutation polynomial 
			$f_2(x)=x+\gamma\operatorname{Tr}_{q}^{q^2}(x+x^2+x^3+x^{q+2})$ over $\mathbb{F}_{q^2}$ is
		$f_2^{-1}(x)=x+\gamma\operatorname{Tr}_{q}^{q^2}(x)+\gamma(\operatorname{Tr}_{q}^{q^2}(x))^2+\gamma(\operatorname{Tr}_{q}^{q^2}(x))^3$.
	\end{thm}
	\begin{proof} We find $f_2\circ T(x)=f_2(T(x))$, where $T(x)$ is same as in the previous theorem. We have,
	\begin{align}
		f_2\circ T(x)=&\operatorname{Tr}_{q}^{q^2}(\beta x)+\alpha\operatorname{Tr}_{q}^{q^2}(\delta x)+\gamma\operatorname{Tr}_{q}^{q^2}[\operatorname{Tr}_{q}^{q^2}(\beta x)+\alpha\operatorname{Tr}_{q}^{q^2}(\delta x)+\{\operatorname{Tr}_{q}^{q^2}(\beta x)\notag\\&+\alpha\operatorname{Tr}_{q}^{q^2}(\delta x)\}^2+\{\operatorname{Tr}_{q}^{q^2}(\beta x)+\alpha\operatorname{Tr}_{q}^{q^2}(\delta x)\}^3+\{\operatorname{Tr}_{q}^{q^2}(\beta x)\notag\\&+\alpha\operatorname{Tr}_{q}^{q^2}(\delta x)\}^{q+2}]\notag\\
		=&\operatorname{Tr}_{q}^{q^2}(\beta x)+\alpha\operatorname{Tr}_{q}^{q^2}(\delta x)+\gamma\operatorname{Tr}_{q}^{q^2}[(\operatorname{Tr}_{q}^{q^2}(\beta x)+\alpha\operatorname{Tr}_{q}^{q^2}(\delta x)+\{\operatorname{Tr}_{q}^{q^2}(\beta x)\notag\\&+\alpha\operatorname{Tr}_{q}^{q^2}(\delta x)\}^2\{1+\operatorname{Tr}_{q}^{q^2}(\beta x)+\alpha\operatorname{Tr}_{q}^{q^2}(\delta x)+\operatorname{Tr}_{q}^{q^2}(\beta x)\notag\\&+\alpha^{q}\operatorname{Tr}_{q}^{q^2}(\delta x)\}]\notag\\
		=&\operatorname{Tr}_{q}^{q^2}(\beta x)+\alpha\operatorname{Tr}_{q}^{q^2}(\delta x)+\gamma\operatorname{Tr}_{q}^{q^2}[\operatorname{Tr}_{q}^{q^2}(\beta x)+\alpha\operatorname{Tr}_{q}^{q^2}(\delta x)+\{\operatorname{Tr}_{q}^{q^2}(\beta x)\notag\\&+\alpha\operatorname{Tr}_{q}^{q^2}(\delta x)\}^2(1+\operatorname{Tr}_{q}^{q^2}(\beta x)+\alpha\operatorname{Tr}_{q}^{q^2}(\delta x)+\operatorname{Tr}_{q}^{q^2}(\beta x)+\alpha\operatorname{Tr}_{q}^{q^2}(\delta x)\notag\\&+\operatorname{Tr}_{q}^{q^2}(\delta x))]\notag\\
		=&\operatorname{Tr}_{q}^{q^2}(\beta x)+\alpha\operatorname{Tr}_{q}^{q^2}(\delta x)+\gamma\{\operatorname{Tr}_{q}^{q^2}(\delta x)+(\operatorname{Tr}_{q}^{q^2}(\delta x))^2+(\operatorname{Tr}_{q}^{q^2}(\delta x))^3\}.\label{2a}
		\end{align}
	
	Now, we consider the equation $f_2\circ T(x)=T(y)=\operatorname{Tr}_{q}^{q^2}(\beta y)+\alpha\operatorname{Tr}_{q}^{q^2}(\delta y)$.
	On comparing the coefficient of $\alpha$ on each sides, we get
	\begin{align*}
		\operatorname{Tr}_{q}^{q^2}(\delta x)=&\operatorname{Tr}_{q}^{q^2}(\delta y)\\
		\operatorname{Tr}_{q}^{q^2}(\beta x)=&\operatorname{Tr}_{q}^{q^2}(\beta y)+\gamma\{\operatorname{Tr}_{q}^{q^2}(\delta y)+(\operatorname{Tr}_{q}^{q^2}(\delta y))^2+(\operatorname{Tr}_{q}^{q^2}(\delta y))^3\}.
	\end{align*}
	Equivalently, we have
	$$\begin{bmatrix}
		\beta&\beta^{q}\\
		\delta&\delta^{q}
	\end{bmatrix}
	\begin{bmatrix}
		x\\x^{q}
	\end{bmatrix}=
	\begin{bmatrix}
		\operatorname{Tr}_{q}^{q^2}(\beta y)+\gamma\{\operatorname{Tr}_{q}^{q^2}(\delta y)+(\operatorname{Tr}_{q}^{q^2}(\delta y))^2+(\operatorname{Tr}_{q}^{q^2}(\delta y))^3\}\\
		\operatorname{Tr}_{q}^{q^2}(\delta y)
	\end{bmatrix}.$$
	Applying elementary row operations $R_1\rightarrow\delta^{q}R_1$ and $R_2\rightarrow\beta^{q}R_2$, we get
	
	$$\begin{bmatrix}
		\beta\delta^{q}&\beta^{q}\delta^{q}\\
		\beta^{q}\delta&\beta^{q}\delta^{q}
	\end{bmatrix}
	\begin{bmatrix}
		x\\x^{q}
	\end{bmatrix}=B$$
	\text{where},$$ B'=
	\begin{bmatrix}
		\delta^{q}\operatorname{Tr}_{q}^{q^2}(\beta y)+\delta^{q}\gamma\{\operatorname{Tr}_{q}^{q^2}(\delta y)+(\operatorname{Tr}_{q}^{q^2}(\delta y))^2+(\operatorname{Tr}_{q}^{q^2}(\delta y))^3\}\\
		\beta^{q}\operatorname{Tr}_{q}^{q^2}(\delta y)
	\end{bmatrix}.$$
	Again applying elementary row operation $R_1\rightarrow R_1+R_2$, we get
	\begin{align}\begin{bmatrix}
		\beta\delta^{q}+\beta^{q}\delta&0\\
		\beta^{q}\delta&\beta^{q}\delta^{q}
	\end{bmatrix}
	\begin{bmatrix}
		x\\x^{q}
	\end{bmatrix}=B\label{2b}\end{align}
	where,$$B=\begin{bmatrix}
		\delta^{q}\operatorname{Tr}_{q}^{q^2}(\beta y)+\delta^{q}\gamma\{\operatorname{Tr}_{q}^{q^2}(\delta y)+(\operatorname{Tr}_{q}^{q^2}(\delta y))^2+(\operatorname{Tr}_{q}^{q^2}(\delta y))^3\}+\beta^{q}\operatorname{Tr}_{q}^{q^2}(\delta y)\\
		\beta^{q}\operatorname{Tr}_{q}^{q^2}(\delta y)
	\end{bmatrix}.$$
		From the equation \ref{2b}, we get
		\begin{align*}
	x=&(\beta\delta^{q}+\beta^{q}\delta)^{-1}\{\delta^{q}\operatorname{Tr}_{q}^{q^2}(\beta y)+\delta^{q}\gamma\operatorname{Tr}_{q}^{q^2}(\delta y)+\delta^{q}\gamma(\operatorname{Tr}_{q}^{q^2}(\delta y))^2\\&+\delta^{q}\gamma(\operatorname{Tr}_{q}^{q^2}(\delta y))^3+\beta^{q}\operatorname{Tr}_{q}^{q^2}(\delta y)\}
	\end{align*}
	or
	\begin{align*}
	T^{-1}\circ f_2^{-1}\circ T(y)=&(\beta\delta^{q}+\beta^{q}\delta)^{-1}\{\delta^{q}\operatorname{Tr}_{q}^{q^2}(\beta y)+\delta^{q}\gamma\operatorname{Tr}_{q}^{q^2}(\delta y)+\delta^{q}\gamma\\&(\operatorname{Tr}_{q}^{q^2}(\delta y))^2+\delta^{q}\gamma(\operatorname{Tr}_{q}^{q^2}(\delta y))^3+\beta^{q}\operatorname{Tr}_{q}^{q^2}(\delta y)\}.
\end{align*}
	Now substituting $y=T^{-1}(x)$ in above equation, we get
	\begin{align*}
		T^{-1}\circ f_2^{-1}(x)=&(\beta\delta^{q}+\beta^{q}\delta)^{-1}\{\delta^{q}\operatorname{Tr}_{q}^{q^2}(\beta T^{-1}(x))+\delta^{q}\gamma\operatorname{Tr}_{q}^{q^2}(\delta T^{-1}(x))\\&+\delta^{q}\gamma(\operatorname{Tr}_{q}^{q^2}(\delta T^{-1}(x)))^2+\delta^{q}\gamma(\operatorname{Tr}_{q}^{q^2}(\delta T^{-1}(x)))^3\\&+\beta^{q}\operatorname{Tr}_{q}^{q^2}(\delta T^{-1}(x))\}.
	\end{align*}
	Now using lemma \ref{lem3}, we get
	\begin{align*}
		T^{-1}\circ f_2^{-1}(x)=&(\beta\delta^{q}+\beta^{q}\delta)^{-1}\{\delta^{q}\operatorname{Tr}_{q}^{q^2}((1+\alpha)x)+\delta^{q}\gamma\operatorname{Tr}_{q}^{q^2}(x)\\&+\delta^{q}\gamma(\operatorname{Tr}_{q}^{q^2}(x))^2+\delta^{q}\gamma(\operatorname{Tr}_{q}^{q^2}(x))^3+\beta^{q}\operatorname{Tr}_{q}^{q^2}(x)\}.
	\end{align*}
	 We note that
	\begin{align}
		f_2^{-1}(x)=&\operatorname{Tr}_{q}^{q^2}[\beta (T^{-1}\circ f_2^{-1}(x))]+\alpha\operatorname{Tr}_{q}^{q^2}[\delta (T^{-1}\circ f_2^{-1}(x))].\label{2c}
	\end{align}
	Now,
	\begin{align}
		\operatorname{Tr}_{q}^{q^2}[\beta (T^{-1}\circ f_2^{-1}(x))]=&\operatorname{Tr}_{q}^{q^2}[\beta (\beta\delta^{q}+\beta^{q}\delta)^{-1}\{(\delta^{q}\operatorname{Tr}_{q}^{q^2}((1+\alpha)x)\notag\\&+\delta^{q}\gamma\operatorname{Tr}_{q}^{q^2}(x)+\delta^{q}\gamma(\operatorname{Tr}_{q}^{q^2}(x))^2+\delta^{q}\gamma(\operatorname{Tr}_{q}^{q^2}(x))^3\notag\\&+\beta^{q}\operatorname{Tr}_{q}^{q^2}(x))\}]\notag\\
		=&\operatorname{Tr}_{q}^{q^2}[\beta\delta^{q}(\beta\delta^{q}+\beta^{q}\delta)^{-1}\operatorname{Tr}_{q}^{q^2}((1+\alpha)x)\notag\\&+\beta\delta^{q}\gamma(\beta\delta^{q}+\beta^{q}\delta)^{-1}\operatorname{Tr}_{q}^{q^2}(x)\notag\\&+\beta\delta^{q}\gamma(\beta\delta^{q}+\beta^{q}\delta)^{-1}(\operatorname{Tr}_{q}^{q^2}(x))^2\notag\\&+\beta\delta^{q}\gamma(\beta\delta^{q}+\beta^{q}\delta)^{-1}(\operatorname{Tr}_{q}^{q^2}(x))^3\notag\\&+\beta\beta^{q}(\beta\delta^{q}+\beta^{q}\delta)^{-1}\operatorname{Tr}_{q}^{q^2}(x)]\notag\\
		=&\operatorname{Tr}_{q}^{q^2}((1+\alpha)x)+\gamma\operatorname{Tr}_{q}^{q^2}(x)+\gamma(\operatorname{Tr}_{q}^{q^2}(x))^2\notag\\&+\gamma(\operatorname{Tr}_{q}^{q^2}(x))^3.\label{2d}
	\end{align}\\
	and
	\begin{align}
		\operatorname{Tr}_{q}^{q^2}[\delta (T^{-1}\circ f_2^{-1}(x))]=&\operatorname{Tr}_{q}^{q^2}[\delta (\beta\delta^{q}+\beta^{q}\delta)^{-1}\{\delta^{q}\operatorname{Tr}_{q}^{q^2}((1+\alpha)x)\notag\\&+\delta^{q}\gamma\operatorname{Tr}_{q}^{q^2}(x)+\delta^{q}\gamma(\operatorname{Tr}_{q}^{q^2}(x))^2+\delta^{q}\gamma(\operatorname{Tr}_{q}^{q^2}(x))^3\notag\\&+\beta^{q}\operatorname{Tr}_{q}^{q^2}(x)\}]\notag\\
		=&\operatorname{Tr}_{q}^{q^2}[\delta\delta^{q}(\beta\delta^{q}+\beta^{q}\delta)^{-1}\operatorname{Tr}_{q}^{q^2}((1+\alpha)x)\notag\\&+\delta\delta^{q}\gamma(\beta\delta^{q}+\beta^{q}\delta)^{-1}\operatorname{Tr}_{q}^{q^2}(x)\notag\\&+\delta\delta^{q}\gamma(\beta\delta^{q}+\beta^{q}\delta)^{-1}(\operatorname{Tr}_{q}^{q^2}(x))^2\notag\\&+\delta\delta^{q}\gamma(\beta\delta^{q}+\beta^{q}\delta)^{-1}(\operatorname{Tr}_{q}^{q^2}(x))^3\notag\\&+\delta\beta^{q}(\beta\delta^{q}+\beta^{q}\delta)^{-1}\operatorname{Tr}_{q}^{q^2}(x)]\notag\\
		=&\operatorname{Tr}_{q}^{q^2}(x).\label{2e}
	\end{align}
	 Now using equations \ref{2d} and \ref{2e} in equation \ref{2c}, we get
	\begin{align*}
		f_2^{-1}(x)=&\operatorname{Tr}_{q}^{q^2}((1+\alpha)x)+\gamma\operatorname{Tr}_{q}^{q^2}(x)+\gamma(\operatorname{Tr}_{q}^{q^2}(x))^2+\gamma(\operatorname{Tr}_{q}^{q^2}(x))^3\\&+\alpha\operatorname{Tr}_{q}^{q^2}(x)\\
		=&x+\gamma\operatorname{Tr}_{q}^{q^2}(x)+\gamma(\operatorname{Tr}_{q}^{q^2}(x))^2+\gamma(\operatorname{Tr}_{q}^{q^2}(x))^3.
	\end{align*}
\end{proof}

\begin{thm} The compositional inverse of permutation polynomial 
		$f_3(x)=x+\gamma\operatorname{Tr}_{q}^{q^2}(x+x^3+x^{q+2})$ over $\mathbb{F}_{q^2}$ is \\
	$f_3^{-1}(x)=\begin{cases}x+\gamma\operatorname{Tr}_{q}^{q^2}(x)+\gamma(\operatorname{Tr}_{q}^{q^2}(x))^3,& \text{if}~\gamma\in\mathbb{F}_q\\
	\operatorname{Tr}_{q}^{q^2}(1+\gamma)x+\gamma(\operatorname{Tr}_{q}^{q^2}(x))^t,& \text{if}~\operatorname{Tr}_{q}^{q^2}(\gamma)=1\text{, } \text{and}~ m \text{ is odd}
	\end{cases}$\\
	where $3t\equiv 1\mod{(2^m-1)}$.
\end{thm}
\begin{proof} We can easily see that $\operatorname{Tr}_{q}^{q^2}(\gamma+\gamma^2)=0$ means either $\operatorname{Tr}_{q}^{q^2}(\gamma)=0$ or $\operatorname{Tr}_{q}^{q^2}(\gamma)=1$. First, we assume that $\operatorname{Tr}_{q}^{q^2}(\gamma)=0$ or equivalently $\gamma \in \mathbb{F}_{q}$. We have,
	\begin{align}
		f_3\circ T(x)=&\operatorname{Tr}_{q}^{q^2}(\beta x)+\alpha\operatorname{Tr}_{q}^{q^2}(\delta x)+\gamma\operatorname{Tr}_{q}^{q^2}[\operatorname{Tr}_{q}^{q^2}(\beta x)+\alpha\operatorname{Tr}_{q}^{q^2}(\delta x)+\{\operatorname{Tr}_{q}^{q^2}(\beta x)\notag\\&+\alpha\operatorname{Tr}_{q}^{q^2}(\delta x)\}^3+\{\operatorname{Tr}_{q}^{q^2}(\beta x)+\alpha\operatorname{Tr}_{q}^{q^2}(\delta x)\}^{q+2}]\notag\\
		=&\operatorname{Tr}_{q}^{q^2}(\beta x)+\alpha\operatorname{Tr}_{q}^{q^2}(\delta x)+\gamma\operatorname{Tr}_{q}^{q^2}[\operatorname{Tr}_{q}^{q^2}(\beta x)+\alpha\operatorname{Tr}_{q}^{q^2}(\delta x)+\{\operatorname{Tr}_{q}^{q^2}(\beta x)\notag\\&+\alpha\operatorname{Tr}_{q}^{q^2}(\delta x)\}^2\{\operatorname{Tr}_{q}^{q^2}(\beta x)+\alpha\operatorname{Tr}_{q}^{q^2}(\delta x)+\operatorname{Tr}_{q}^{q^2}(\beta x)+\alpha\operatorname{Tr}_{q}^{q^2}(\delta x)\notag\\&+\operatorname{Tr}_{q}^{q^2}(\delta x)\}]\notag\\
		=&\operatorname{Tr}_{q}^{q^2}(\beta x)+\alpha\operatorname{Tr}_{q}^{q^2}(\delta x)+\gamma\{\operatorname{Tr}_{q}^{q^2}(\delta x)+(\operatorname{Tr}_{q}^{q^2}(\delta x))^3\}.\label{3a}
	\end{align}
	On considering the equation $f_3 \circ T(x)=T(y)=\operatorname{Tr}_{q}^{q^2}(\beta y)+\alpha\operatorname{Tr}_{q}^{q^2}(\delta y)$, we get the following system of equations
	\begin{align*}
		\operatorname{Tr}_{q}^{q^2}(\delta x)=&\operatorname{Tr}_{q}^{q^2}(\delta y)\\
		\operatorname{Tr}_{q}^{q^2}(\beta x)=&\operatorname{Tr}_{q}^{q^2}(\beta y)+\gamma\{\operatorname{Tr}_{q}^{q^2}(\delta y)+(\operatorname{Tr}_{q}^{q^2}(\delta y))^3\}.
	\end{align*}
	Equivalently in matrix form we get,
	$$\begin{bmatrix}
		\beta&\beta^{q}\\
		\delta&\delta^{q}
	\end{bmatrix}
	\begin{bmatrix}
		x\\x^{q}
	\end{bmatrix}=
	\begin{bmatrix}
		\operatorname{Tr}_{q}^{q^2}(\beta y)+\gamma\{\operatorname{Tr}_{q}^{q^2}(\delta y)+(\operatorname{Tr}_{q}^{q^2}(\delta y))^3\}\\
		\operatorname{Tr}_{q}^{q^2}(\delta y)
	\end{bmatrix}.$$
	Applying elementary row operations $R_1\rightarrow\delta^{q}R_1$ and $R_2\rightarrow\beta^{q}R_2$, we get
	$$\begin{bmatrix}
		\beta\delta^{q}&\beta^{q}\delta^{q}\\
		\beta^{q}\delta&\beta^{q}\delta^{q}
	\end{bmatrix}
	\begin{bmatrix}
		x\\x^{q}
	\end{bmatrix}=
	\begin{bmatrix}
		\delta^{q}\operatorname{Tr}_{q}^{q^2}(\beta y)+\delta^{q}\gamma\{\operatorname{Tr}_{q}^{q^2}(\delta y)+(\operatorname{Tr}_{q}^{q^2}(\delta y))^3\}\\
		\beta^{q}\operatorname{Tr}_{q}^{q^2}(\delta y)
	\end{bmatrix}.$$
	Again applying elementary row operations $R_1\rightarrow R_1+R_2$, we get
	\begin{align}\begin{bmatrix}
		\beta\delta^{q}+\beta^{q}\delta&0\\
		\beta^{q}\delta&\beta^{q}\delta^{q}
	\end{bmatrix}
	\begin{bmatrix}
		x\\x^{q}
	\end{bmatrix}=C\label{3b}\end{align}
	\text{where}, $C=
	\begin{bmatrix}
		\delta^{q}\operatorname{Tr}_{q}^{q^2}(\beta y)+\delta^{q}\gamma\{\operatorname{Tr}_{q}^{q^2}(\delta y)+(\operatorname{Tr}_{q}^{q^2}(\delta y))^3\}+\beta^{q}\operatorname{Tr}_{q}^{q^2}(\delta y)\\
		\beta^{q}\operatorname{Tr}_{q}^{q^2}(\delta y)
	\end{bmatrix}.$\\
		From the equation \ref{3b}, we get
		\begin{align*}
	x=&(\beta\delta^{q}+\beta^{q}\delta)^{-1}\{\delta^{q}\operatorname{Tr}_{q}^{q^2}(\beta y)+\delta^{q}\gamma\operatorname{Tr}_{q}^{q^2}(\delta y)+\delta^{q}\gamma(\operatorname{Tr}_{q}^{q^2}(\delta y))^3\\&+\beta^{q}\operatorname{Tr}_{q}^{q^2}(\delta y)\}
\end{align*}
	or
	\begin{align*}
	T^{-1}\circ f_3^{-1}\circ T(y)=&(\beta\delta^{q}+\beta^{q}\delta)^{-1}\{\delta^{q}\operatorname{Tr}_{q}^{q^2}(\beta y)+\delta^{q}\gamma\operatorname{Tr}_{q}^{q^2}(\delta y)+\delta^{q}\gamma\\&(\operatorname{Tr}_{q}^{q^2}(\delta y))^3+\beta^{q}\operatorname{Tr}_{q}^{q^2}(\delta y)\}.
\end{align*}
	We replace $y$ with $T^{-1}(x)$ in above equation
	\begin{align*}
		T^{-1}\circ f_3^{-1}(x)=&(\beta\delta^{q}+\beta^{q}\delta)^{-1}\{\delta^{q}\operatorname{Tr}_{q}^{q^2}(\beta T^{-1}(x))+\delta^{q}\gamma\operatorname{Tr}_{q}^{q^2}(\delta T^{-1}(x))\\&+\delta^{q}\gamma(\operatorname{Tr}_{q}^{q^2}(\delta T^{-1}(x)))^3+\beta^{q}\operatorname{Tr}_{q}^{q^2}(\delta T^{-1}(x))\}.
	\end{align*}
	Now using lemma \ref{lem3} we get
	\begin{align*}
		T^{-1}\circ f_3^{-1}(x)=&(\beta\delta^{q}+\beta^{q}\delta)^{-1}\{\delta^{q}\operatorname{Tr}_{q}^{q^2}((1+\alpha)x)+\delta^{q}\gamma\operatorname{Tr}_{q}^{q^2}(x)\\&+\delta^{q}\gamma(\operatorname{Tr}_{q}^{q^2}(x))^3+\beta^{q}\operatorname{Tr}_{q}^{q^2}(x)\}.
	\end{align*}
	 We know that
	\begin{align}
		f_3^{-1}(x)=&\operatorname{Tr}_{q}^{q^2}[\beta (T^{-1}\circ f_3^{-1}(x))]+\alpha\operatorname{Tr}_{q}^{q^2}[\delta (T^{-1}\circ f_3^{-1}(x))].
		\label{3c}
	\end{align}
		Now we have
	\begin{align}
		\operatorname{Tr}_{q}^{q^2}[\beta (T^{-1}\circ f_3^{-1}(x))]=&\operatorname{Tr}_{q}^{q^2}[\beta (\beta\delta^{q}+\beta^{q}\delta)^{-1}\{\delta^{q}\operatorname{Tr}_{q}^{q^2}((1+\alpha)x)\notag\\&+\delta^{q}\gamma\operatorname{Tr}_{q}^{q^2}(x)+\delta^{q}\gamma(\operatorname{Tr}_{q}^{q^2}(x))^3+\beta^{q}\operatorname{Tr}_{q}^{q^2}(x)\}]\notag\\
		=&\operatorname{Tr}_{q}^{q^2}[\beta\delta^{q}(\beta\delta^{q}+\beta^{q}\delta)^{-1}\operatorname{Tr}_{q}^{q^2}((1+\alpha)x)\notag\\&+\beta\delta^{q}\gamma(\beta\delta^{q}+\beta^{q}\delta)^{-1}\operatorname{Tr}_{q}^{q^2}(x)\notag\\&+\beta\delta^{q}\gamma(\beta\delta^{q}+\beta^{q}\delta)^{-1}(\operatorname{Tr}_{q}^{q^2}(x))^3\notag\\&+\beta\beta^{q}(\beta\delta^{q}+\beta^{q}\delta)^{-1}\operatorname{Tr}_{q}^{q^2}(x)]\notag\\
		=&\operatorname{Tr}_{q}^{q^2}((1+\alpha)x)+\gamma\operatorname{Tr}_{q}^{q^2}(x)+\gamma(\operatorname{Tr}_{q}^{q^2}(x))^3.\label{3d}
	\end{align}
	and
	\begin{align}
		\operatorname{Tr}_{q}^{q^2}[\delta (T^{-1}\circ f_3^{-1}(x))]=&\operatorname{Tr}_{q}^{q^2}[\delta (\beta\delta^{q}+\beta^{q}\delta)^{-1}\{\delta^{q}\operatorname{Tr}_{q}^{q^2}((1+\alpha)x)\notag\\&+\delta^{q}\gamma\operatorname{Tr}_{q}^{q^2}(x)+\delta^{q}\gamma(\operatorname{Tr}_{q}^{q^2}(x))^3+\beta^{q}\operatorname{Tr}_{q}^{q^2}(x)\}]\notag\\
		=&\operatorname{Tr}_{q}^{q^2}[\delta\delta^{q}(\beta\delta^{q}+\beta^{q}\delta)^{-1}\operatorname{Tr}_{q}^{q^2}((1+\alpha)x)\notag\\&+\delta\delta^{q}\gamma(\beta\delta^{q}+\beta^{q}\delta)^{-1}\operatorname{Tr}_{q}^{q^2}(x)\notag\\&+\delta\delta^{q}\gamma(\beta\delta^{q}+\beta^{q}\delta)^{-1}(\operatorname{Tr}_{q}^{q^2}(x))^3\notag\\&+\delta\beta^{q}(\beta\delta^{q}+\beta^{q}\delta)^{-1}\operatorname{Tr}_{q}^{q^2}(x)]\notag\\
		=&\operatorname{Tr}_{q}^{q^2}(x).\label{3e}
	\end{align}
	Now using equations \ref{3d} and \ref{3e} in equation \ref{3c}, we get
	\begin{align*}
		f_3^{-1}(x)=&\operatorname{Tr}_{q}^{q^2}((1+\alpha)x)+\gamma\operatorname{Tr}_{q}^{q^2}(x)+\gamma(\operatorname{Tr}_{q}^{q^2}(x))^3+\alpha\operatorname{Tr}_{q}^{q^2}(x)\\
		=&x+\gamma\operatorname{Tr}_{q}^{q^2}(x)+\gamma(\operatorname{Tr}_{q}^{q^2}(x))^3.
	\end{align*}
	
	Next we assume $\operatorname{Tr}(\gamma)=1$. Clearly, $\gamma\in \mathbb{F}_{q^2}\setminus\mathbb{F}_q$ and 
	$\{1,\gamma\}$ is a basis of $\mathbb{F}_{q^2}$ over $\mathbb{F}_{q}$.
	Let $U(x)=\operatorname{Tr}_q^{q^2}{(\beta x)}+\gamma\operatorname{Tr}_q^{q^2}(\delta x)$. By lemma \ref{lem2}, we have
	$U^{-1}(x)=(\beta\delta^{q}+\delta\beta^{q})^{-1}\{\delta^{q}\operatorname{Tr}_{q}^{q^2}((1+\gamma)x)+\beta^{q}\operatorname{Tr}_{q}^{q^2}(x)\}.$
	Using equation \ref{3a}, we can easily see that
  \begin{align}
  	f_3(U(x))
  	=&\operatorname{Tr}_q^{q^2}(\beta x)+\gamma(\operatorname{Tr}_q^{q^2}(\delta x))^3.\label{3f}
  \end{align}
	Next we consider the equation $f_3\circ U(x)=U(y)=\operatorname{Tr}_{q}^{q^2}(\beta y)+\gamma\operatorname{Tr}_{q}^{q^2}(\delta y)$. This gives the following system of equations	
	\begin{align*}
		\operatorname{Tr}_q^{q^2}(\beta x)=\operatorname{Tr}_q^{q^2}(\beta y)\\
		\operatorname{Tr}_q^{q^2}(\delta x)^3=\operatorname{Tr}_q^{q^2}(\delta y).
	\end{align*}
	Let $t$ be a positive integer satisfying $3t\equiv 1\mod{(2^m-1)}$.
	Then we have $\operatorname{Tr}_q^{q^2}{\delta x}=(\operatorname{Tr}_q^{q^2}(\delta y))^t$. Equivalently, we have the following system of equation
	$$\begin{bmatrix}
		\beta&\beta^{q}\\
		\delta&\delta^{q}
	\end{bmatrix}
	\begin{bmatrix}
		x\\x^{q}
	\end{bmatrix}=
	\begin{bmatrix}
		\operatorname{Tr}_{q}^{q^2}(\beta y)\\
		(\operatorname{Tr}_{q}^{q^2}(\delta y))^t
	\end{bmatrix}.$$
	Applying elementary row operation $R_1\rightarrow\delta^{q}R_1$ and $R_2\rightarrow\beta^{q}R_2$, we get
	$$\begin{bmatrix}
		\beta\delta^{q}&\beta^{q}\delta^{q}\\
		\beta^{q}\delta&\beta^{q}\delta^{q}
	\end{bmatrix}
	\begin{bmatrix}
		x\\x^{q}
	\end{bmatrix}=
	\begin{bmatrix}
		\delta^{q}\operatorname{Tr}_{q}^{q^2}(\beta y)\\
		\beta^{q}(\operatorname{Tr}_{q}^{q^2}(\delta y))^t
	\end{bmatrix}.$$
	Again applying elementary row operation $R_1\rightarrow R_1+R_2$, we get
	\begin{align}\begin{bmatrix}
			\beta\delta^{q}+\beta^{q}\delta&0\\
			\beta^{q}\delta&\beta^{q}\delta^{q}
		\end{bmatrix}
		\begin{bmatrix}
			x\\x^{q}
		\end{bmatrix}=
	\begin{bmatrix}
		\delta^{q}\operatorname{Tr}_{q}^{q^2}(\beta y)+	\beta^{q}(\operatorname{Tr}_{q}^{q^2}(\delta y))^t\\
		\beta^{q}(\operatorname{Tr}_{q}^{q^2}(\delta y))^t
	\end{bmatrix}.\label{3g}\end{align}
	From the equation \ref{3g}, we get
	\begin{align*}
		x=&(\beta\delta^{q}+\beta^{q}\delta)^{-1}\{\delta^{q}\operatorname{Tr}_{q}^{q^2}(\beta y)+\beta^{q}(\operatorname{Tr}_{q}^{q^2}(\delta y))^t\}
	\end{align*}
	or\begin{align*}
		U^{-1}\circ f_3^{-1}\circ U(y)=&(\beta\delta^{q}+\beta^{q}\delta)^{-1}\{\delta^{q}\operatorname{Tr}_{q}^{q^2}(\beta y)+\beta^{q}(\operatorname{Tr}_{q}^{q^2}(\delta y))^t\}.
	\end{align*}
	For $y=U^{-1}(x)$, we get
	\begin{align*}
		U^{-1}\circ f_3^{-1}(x)=&(\beta\delta^{q}+\beta^{q}\delta)^{-1}\{\delta^{q}\operatorname{Tr}_{q}^{q^2}(\beta U^{-1}(x))+\beta^{q}(\operatorname{Tr}_{q}^{q^2}(\delta U^{-1}(x)))^t\}.
	\end{align*}
	Now using Lemma \ref{lem3}, we get
	\begin{align*}
		U^{-1}\circ f_3^{-1}(x)=&(\beta\delta^{q}+\beta^{q}\delta)^{-1}\{\delta^{q}\operatorname{Tr}_{q}^{q^2}((1+\gamma)x)+\beta^{q}(\operatorname{Tr}_{q}^{q^2}(x))^t\}.
	\end{align*}
	Taking each side composition with $U$, we get
	\begin{align*}
		f_3^{-1}(x)=&\operatorname{Tr}_{q}^{q^2}[\beta (\beta\delta^{q}+\beta^{q}\delta)^{-1}\{\delta^{q}\operatorname{Tr}_{q}^{q^2}((1+\gamma)x)+\beta^{q}(\operatorname{Tr}_{q}^{q^2}(x))^t\}]\\&+\gamma\operatorname{Tr}_{q}^{q^2}[\delta (\beta\delta^{q}+\beta^{q}\delta)^{-1}\{\delta^{q}\operatorname{Tr}_{q}^{q^2}((1+\gamma)x)+\beta^{q}(\operatorname{Tr}_{q}^{q^2}(x))^t\}]\\
		=&\operatorname{Tr}_{q}^{q^2}((1+\gamma)x)+\gamma(\operatorname{Tr}_{q}^{q^2}(x))^t.
	\end{align*}
\end{proof}

\begin{thm} The compositional inverse of permutation polynomial 
	$f_4(x)=x+\gamma\operatorname{Tr}_{q}^{q^2}(x^2+x^3+x^{q+2})$ over $\mathbb{F}_{q^2}$ is \\
	$f_4^{-1}(x)=\begin{cases}x+\gamma(\operatorname{Tr}_{q}^{q^2}(x))^2+\gamma(\operatorname{Tr}_{q}^{q^2}(x))^3&,\text{if } \gamma\in\mathbb{F}_q \\
	\operatorname{Tr}_{q}^{q^2}((1+\gamma)x)+\gamma[(\operatorname{Tr}_{q}^{q^2}(x)+1)^t+1]&, \text{if}~\operatorname{Tr}_{q}^{q^2}(\gamma)=1\text{, } \text{and}~ m \text{ is odd}
	\end{cases}$
	where $3t\equiv 1\mod{(2^m-1)}$.
\end{thm}
\begin{proof}
First we assume that $\gamma \in \mathbb{F}_{q}$.	
	We have
\begin{align}
	f_4\circ T(x)=&\operatorname{Tr}_{q}^{q^2}(\beta x)+\alpha\operatorname{Tr}_{q}^{q^2}(\delta x)+\gamma\operatorname{Tr}_{q}^{q^2}[\{\operatorname{Tr}_{q}^{q^2}(\beta x)+\alpha\operatorname{Tr}_{q}^{q^2}(\delta x)\}^2\notag\\&+\{\operatorname{Tr}_{q}^{q^2}(\beta x)+\alpha\operatorname{Tr}_{q}^{q^2}(\delta x)\}^3+\{\operatorname{Tr}_{q}^{q^2}(\beta x)+\alpha\operatorname{Tr}_{q}^{q^2}(\delta x)\}^{q+2}]\notag\\
	=&\operatorname{Tr}_{q}^{q^2}(\beta x)+\alpha\operatorname{Tr}_{q}^{q^2}(\delta x)+\gamma\operatorname{Tr}_{q}^{q^2}[\{\operatorname{Tr}_{q}^{q^2}(\beta x)+\alpha\operatorname{Tr}_{q}^{q^2}(\delta x)\}^2\notag\\&\{1+\operatorname{Tr}_{q}^{q^2}(\beta x)+\alpha\operatorname{Tr}_{q}^{q^2}(\delta x)+(\operatorname{Tr}_{q}^{q^2}(\beta x)+\alpha\operatorname{Tr}_{q}^{q^2}(\delta x))^q\}]\notag\\
	=&\operatorname{Tr}_{q}^{q^2}(\beta x)+\alpha\operatorname{Tr}_{q}^{q^2}(\delta x)+\gamma\operatorname{Tr}_{q}^{q^2}[\{\operatorname{Tr}_{q}^{q^2}(\beta x)+\alpha\operatorname{Tr}_{q}^{q^2}(\delta x)\}^2\notag\\&\{1+\operatorname{Tr}_{q}^{q^2}(\beta x)+\alpha\operatorname{Tr}_{q}^{q^2}(\delta x)+\operatorname{Tr}_{q}^{q^2}(\beta x)+\alpha\operatorname{Tr}_{q}^{q^2}(\delta x)+\operatorname{Tr}_{q}^{q^2}(\delta x)\}]\notag\\
	=&\operatorname{Tr}_{q}^{q^2}(\beta x)+\alpha\operatorname{Tr}_{q}^{q^2}(\delta x)+\gamma\{(\operatorname{Tr}_{q}^{q^2}(\delta x))^2+(\operatorname{Tr}_{q}^{q^2}(\delta x))^3\}.\label{4a}
\end{align}

Again we consider the equation $f_4\circ T(x)=T(y)=\operatorname{Tr}_{q}^{q^2}(\beta y)+\alpha\operatorname{Tr}_{q}^{q^2}(\delta y)$, we get the following system of equations
\begin{align*}
	\operatorname{Tr}_{q}^{q^2}(\delta x)=&\operatorname{Tr}_{q}^{q^2}(\delta y)\\
	\operatorname{Tr}_{q}^{q^2}(\beta x)=&\operatorname{Tr}_{q}^{q^2}(\beta y)+\gamma\{(\operatorname{Tr}_{q}^{q^2}(\delta y))^2+(\operatorname{Tr}_{q}^{q^2}(\delta y))^3\}.
\end{align*}
In matrix form, we have
$$\begin{bmatrix}
	\beta&\beta^{q}\\
	\delta&\delta^{q}
\end{bmatrix}
\begin{bmatrix}
	x\\x^{q}
\end{bmatrix}=
\begin{bmatrix}
	\operatorname{Tr}_{q}^{q^2}(\beta y)+\gamma\{(\operatorname{Tr}_{q}^{q^2}(\delta y))^2+(\operatorname{Tr}_{q}^{q^2}(\delta y))^3\}\\
	\operatorname{Tr}_{q}^{q^2}(\delta y)
\end{bmatrix}.$$
Applying elementary row operation $R_1\rightarrow\delta^{q}R_1$ and $R_2\rightarrow\beta^{q}R_2$, we get
$$\begin{bmatrix}
	\beta\delta^{q}&\beta^{q}\delta^{q}\\
	\beta^{q}\delta&\beta^{q}\delta^{q}
\end{bmatrix}
\begin{bmatrix}
	x\\x^{q}
\end{bmatrix}=
\begin{bmatrix}
	\delta^{q}\operatorname{Tr}_{q}^{q^2}(\beta y)+\delta^{q}\gamma\{(\operatorname{Tr}_{q}^{q^2}(\delta y))^2+(\operatorname{Tr}_{q}^{q^2}(\delta y))^3\}\\
	\beta^{q}\operatorname{Tr}_{q}^{q^2}(\delta y)
\end{bmatrix}.$$
Now applying $R_1\rightarrow R_1+R_2$, we get
\begin{align}\begin{bmatrix}
	\beta\delta^{q}+\beta^{q}\delta&0\\
	\beta^{q}\delta&\beta^{q}\delta^{q}
\end{bmatrix}
\begin{bmatrix}
	x\\x^{q}
\end{bmatrix}= D\label{4b}\end{align}
where,
$$D=
\begin{bmatrix}
	\delta^{q}\operatorname{Tr}_{q}^{q^2}(\beta y)+\delta^{q}\gamma\{(\operatorname{Tr}_{q}^{q^2}(\delta y))^2+(\operatorname{Tr}_{q}^{q^2}(\delta y))^3\}+\beta^{q}\operatorname{Tr}_{q}^{q^2}(\delta y)\\
	\beta^{q}\operatorname{Tr}_{q}^{q^2}(\delta y)
\end{bmatrix}.$$
From the equation \ref{4b}, we get
\begin{align*}
x=T^{-1}\circ f_4^{-1}\circ T(y)=&(\beta\delta^{q}+\beta^{q}\delta)^{-1}\{\delta^{q}\operatorname{Tr}_{q}^{q^2}(\beta y)+\delta^{q}\gamma(\operatorname{Tr}_{q}^{q^2}(\delta y))^2\\&+\delta^{q}\gamma(\operatorname{Tr}_{q}^{q^2}(\delta y))^3+\beta^{q}\operatorname{Tr}_{q}^{q^2}(\delta y)\}.
\end{align*}
Now taking $y=T^{-1}(x)$ in above equation and using Lemma \ref{lem3} , we can easily see that 
\begin{align*}
	T^{-1}\circ f_4^{-1}(x)=&(\beta\delta^{q}+\beta^{q}\delta)^{-1}\{\delta^{q}\operatorname{Tr}_{q}^{q^2}((1+\alpha)x)+\delta^{q}\gamma(\operatorname{Tr}_{q}^{q^2}(x))^2\\&+\delta^{q}\gamma(\operatorname{Tr}_{q}^{q^2}(x))^3+\beta^{q}\operatorname{Tr}_{q}^{q^2}(x)\}.
\end{align*}
We know that
\begin{align}
	f_4^{-1}(x)=&\operatorname{Tr}_{q}^{q^2}[\beta (T^{-1}\circ f_4^{-1}(x))]+\alpha\operatorname{Tr}_{q}^{q^2}[\delta (T^{-1}\circ f_4^{-1}(x))].\label{4c}
\end{align}
Next we find
\begin{align}
	\operatorname{Tr}_{q}^{q^2}[\beta (T^{-1}\circ f_4^{-1}(x))]=&\operatorname{Tr}_{q}^{q^2}[\beta (\beta\delta^{q}+\beta^{q}\delta)^{-1}\{\delta^{q}\operatorname{Tr}_{q}^{q^2}((1+\alpha)x)+\notag\\&\delta^{q}\gamma(\operatorname{Tr}_{q}^{q^2}(x))^2+\delta^{q}\gamma(\operatorname{Tr}_{q}^{q^2}(x))^3+\beta^{q}\operatorname{Tr}_{q}^{q^2}(x)\}]\notag\\
	=&\operatorname{Tr}_{q}^{q^2}[\beta\delta^{q}(\beta\delta^{q}+\beta^{q}\delta)^{-1}\operatorname{Tr}_{q}^{q^2}((1+\alpha)x)\notag\\&+\beta\delta^{q}\gamma(\beta\delta^{q}+\beta^{q}\delta)^{-1}(\operatorname{Tr}_{q}^{q^2}(x))^2\notag\\&+\beta\delta^{q}\gamma(\beta\delta^{q}+\beta^{q}\delta)^{-1}(\operatorname{Tr}_{q}^{q^2}(x))^3\notag\\&+\beta\beta^{q}(\beta\delta^{q}+\beta^{q}\delta)^{-1}\operatorname{Tr}_{q}^{q^2}(x)]\notag\\
	=&\operatorname{Tr}_{q}^{q^2}((1+\alpha)x)+\gamma(\operatorname{Tr}_{q}^{q^2}(x))^2+\gamma(\operatorname{Tr}_{q}^{q^2}(x))^3.\label{4d}
\end{align}
and
\begin{align}
	\operatorname{Tr}_{q}^{q^2}[\delta (T^{-1}\circ f_4^{-1}(x))]=&\operatorname{Tr}_{q}^{q^2}[\delta (\beta\delta^{q}+\beta^{q}\delta)^{-1}\{\delta^{q}\operatorname{Tr}_{q}^{q^2}((1+\alpha)x)+\notag\\&\delta^{q}\gamma(\operatorname{Tr}_{q}^{q^2}(x))^2+\delta^{q}\gamma(\operatorname{Tr}_{q}^{q^2}(x))^3+\beta^{q}\operatorname{Tr}_{q}^{q^2}(x)\}]\notag\\
	=&\operatorname{Tr}_{q}^{q^2}[\delta\delta^{q}(\beta\delta^{q}+\beta^{q}\delta)^{-1}\operatorname{Tr}_{q}^{q^2}((1+\alpha)x)\notag\\&+\delta\delta^{q}\gamma(\beta\delta^{q}+\beta^{q}\delta)^{-1}(\operatorname{Tr}_{q}^{q^2}(x))^2\notag\\&+\delta\delta^{q}\gamma(\beta\delta^{q}+\beta^{q}\delta)^{-1}(\operatorname{Tr}_{q}^{q^2}(x))^3\notag\\&+\delta\beta^{q}(\beta\delta^{q}+\beta^{q}\delta)^{-1}\operatorname{Tr}_{q}^{q^2}(x)]\notag\\
	=&\operatorname{Tr}_{q}^{q^2}(x).\label{4e}
\end{align}
Now using equations \ref{4d} and \ref{4e} in equation \ref{4c}, we get
\begin{align*}
	f_4^{-1}(x)=&\operatorname{Tr}_{q}^{q^2}((1+\alpha)x)+\gamma(\operatorname{Tr}_{q}^{q^2}(x))^2+\gamma(\operatorname{Tr}_{q}^{q^2}(x))^3+\alpha\operatorname{Tr}_{q}^{q^2}(x)\\
	=&x+\gamma(\operatorname{Tr}_{q}^{q^2}(x))^2+\gamma(\operatorname{Tr}_{q}^{q^2}(x))^3.
\end{align*}
Next we assume $\operatorname{Tr}(\gamma)=1$. We see from equation \ref{4a} that
\begin{align}
	f_4\circ U(x)=&\operatorname{Tr}_q^{q^2}(\beta x)+\gamma\{\operatorname{Tr}_q^{q^2}(\delta x)+\operatorname{Tr}_q^{q^2}(\delta x)^2+\operatorname{Tr}_q^{q^2}(\delta x)^3\}.\label{4f}
\end{align}
From the equation $f_4 \circ U(x)=U(y)$, we get the following system of equations
\begin{align*}
	\operatorname{Tr}_q^{q^2}(\beta x)&=\operatorname{Tr}_q^{q^2}(\beta y)\\
	\operatorname{Tr}_q^{q^2}(\delta x)+\operatorname{Tr}_q^{q^2}(\delta x)^2+\operatorname{Tr}_q^{q^2}(\delta x)^3&=\operatorname{Tr}_q^{q^2}(\delta y).
\end{align*}
We see that $\operatorname{Tr}_q^{q^2}(\delta x)+\operatorname{Tr}_q^{q^2}(\delta x)^2+\operatorname{Tr}_q^{q^2}(\delta x)^3=(\operatorname{Tr}_q^{q^2}(\delta x)+1)^3+1$.
Let $t$ be a positive integer satisfying $3t\equiv 1\mod{(2^m-1)}$. In view of this, we get the following equivalent system of equations
\begin{align}
	\operatorname{Tr}_q^{q^2}(\beta x)&=\operatorname{Tr}_q^{q^2}(\beta y)\label{4g}\\
	\operatorname{Tr}_q^{q^2}(\delta x)&=1+(\operatorname{Tr}_q^{q^2}(\delta y)+1)^{t}.\label{4h}
\end{align}
Equivalently in matrix form, we have
$$\begin{bmatrix}
	\beta&\beta^{q}\\
	\delta&\delta^{q}
\end{bmatrix}
\begin{bmatrix}
	x\\x^{q}
\end{bmatrix}=
\begin{bmatrix}
	\operatorname{Tr}_{q}^{q^2}(\beta y)\\
	(\operatorname{Tr}_q^{q^2}(\delta y)+1)^t+1
\end{bmatrix}.$$
Applying elementary row operations $R_1\rightarrow\delta^{q}R_1$ and $R_2\rightarrow\beta^{q}R_2$, we get
$$\begin{bmatrix}
	\beta\delta^{q}&\beta^{q}\delta^{q}\\
	\beta^{q}\delta&\beta^{q}\delta^{q}
\end{bmatrix}
\begin{bmatrix}
	x\\x^{q}
\end{bmatrix}=
\begin{bmatrix}
	\delta^{q}\operatorname{Tr}_{q}^{q^2}(\beta y)\\
	\beta^{q}\{(\operatorname{Tr}_q^{q^2}(\delta y)+1)^t+1\}
\end{bmatrix}.$$
Again applying $R_1\rightarrow R_1+R_2$, we get
\begin{align}\begin{bmatrix}
		\beta\delta^{q}+\beta^{q}\delta&0\\
		\beta^{q}\delta&\beta^{q}\delta^{q}
	\end{bmatrix}
	\begin{bmatrix}
		x\\x^{q}
	\end{bmatrix}=
	\begin{bmatrix}
		\delta^{q}\operatorname{Tr}_{q}^{q^2}(\beta y)+	\beta^{q}\{(\operatorname{Tr}_q^{q^2}(\delta y)+1)^t+1\}\\
		\beta^{q}\{(\operatorname{Tr}_q^{q^2}(\delta y)+1)^t+1\}
	\end{bmatrix}.\label{4j}\end{align}
From the equation \ref{4j}, we get
\begin{align*}
	x=&(\beta\delta^{q}+\beta^{q}\delta)^{-1}[\delta^{q}\operatorname{Tr}_{q}^{q^2}(\beta y)+\beta^{q}\{(\operatorname{Tr}_q^{q^2}(\delta y)+1)^t+1\}]
\end{align*}
or
\begin{align*}
	U^{-1}\circ f_4^{-1}\circ U(y)=&(\beta\delta^{q}+\beta^{q}\delta)^{-1}[\delta^{q}\operatorname{Tr}_{q}^{q^2}(\beta y)+\beta^{q}\{(\operatorname{Tr}_q^{q^2}(\delta y)+1)^t+1\}].
	\end{align*}
Substituting $y=U^{-1}(x)$ in above equation, we get
\begin{align*}
	U^{-1}\circ f_4^{-1}(x)=&(\beta\delta^{q}+\beta^{q}\delta)^{-1}[\delta^{q}\operatorname{Tr}_{q}^{q^2}(\beta U^{-1}(x))+\beta^{q}\{(\operatorname{Tr}_q^{q^2}(\delta U^{-1}(x))+1)^t+1\}].
\end{align*}
Now using lemma \ref{lem3}, we get
\begin{align*}
	U^{-1}\circ f_4^{-1}(x)=&(\beta\delta^{q}+\beta^{q}\delta)^{-1}[\delta^{q}\operatorname{Tr}_{q}^{q^2}((1+\gamma)x)+\beta^{q}\{(\operatorname{Tr}_{q}^{q^2}(x)+1)^t+1\}].
\end{align*}
Taking each side composition with $U$, we get
\begin{align*}
	f_4^{-1}(x)=&\operatorname{Tr}_{q}^{q^2}((1+\gamma)x)+\gamma[(\operatorname{Tr}_{q}^{q^2}(x)+1)^t+1].
\end{align*}
\end{proof}

In the next two theorems, we have removed the detailed proofs and
present only the key components necessary to continue the arguments.

\begin{thm}Let $q=2^m$, $m$ an odd number and $\gamma\in \mathbb{F}_{2^2}\setminus \mathbb{F}_2$. Then the compositional inverse of permutation polynomial 
	$f_5(x)=x+\gamma\operatorname{Tr}_{q}^{q^2}(x^2+x^{q+2})$ over $\mathbb{F}_{q^2}$ is
	$f_5^{-1}(x)=\operatorname{Tr}_q^{q^2}((1+\gamma)x)+\gamma[\{\operatorname{Tr}_q^{q^2}(x)+(\operatorname{Tr}_q^{q^2}((1+\gamma)x))^3+(\operatorname{Tr}_q^{q^2}((1+\gamma)x))^2+\operatorname{Tr}_q^{q^2}((1+\gamma)x)+1\}^t+\operatorname{Tr}_q^{q^2}((1+\gamma)x)+1]$, where $3t\equiv 1\mod{(2^m-1)}$.
\end{thm}
\begin{proof}
	Since $\gamma\in \mathbb{F}_{2^2}\setminus \mathbb{F}_2$, therefore we have $\gamma^2+\gamma+1=0$, and $\gamma^{2^{m}}=\gamma+1$. We note that for an odd integer $m$, the polynomial $x^2+x+1$ is irreducible over finite fields $\mathbb{F}_{2^{m}}$. Therefore, $\gamma \in \mathbb{F}_{2^{2m}}\setminus \mathbb{F}_{2^m}$, and $\{1, \gamma\}$ is a basis of $\mathbb{F}_{2^{2m}}$ over $\mathbb{F}_{2^{m}}$.
Now we have
	\begin{align}
		f_5\circ U(x) =&\operatorname{Tr}_{q}^{q^2}(\beta x)+\gamma\operatorname{Tr}_{q}^{q^2}(\delta x)+\gamma\operatorname{Tr}_{q}^{q^2}[\{\operatorname{Tr}_{q}^{q^2}(\beta x)+\gamma\operatorname{Tr}_{q}^{q^2}(\delta x)\}^2\notag\\&\{1+\operatorname{Tr}_{q}^{q^2}(\beta x)+\gamma\operatorname{Tr}_{q}^{q^2}(\delta x)+\operatorname{Tr}_{q}^{q^2}(\delta x)\}]\notag\\
		=&\operatorname{Tr}_{q}^{q^2}(\beta x)+\gamma\operatorname{Tr}_{q}^{q^2}(\delta x)+\gamma\operatorname{Tr}_{q}^{q^2}[\{(\operatorname{Tr}_{q}^{q^2}(\beta x))^2+\gamma(\operatorname{Tr}_{q}^{q^2}(\delta x))^2\notag\\&+(\operatorname{Tr}_{q}^{q^2}(\delta x))^2\}\{1+\operatorname{Tr}_{q}^{q^2}(\beta x)+\gamma\operatorname{Tr}_{q}^{q^2}(\delta x)+\operatorname{Tr}_{q}^{q^2}(\delta x)\}]\notag\\
		=&\operatorname{Tr}_{q}^{q^2}(\beta x)+\gamma\operatorname{Tr}_{q}^{q^2}(\delta x)+\gamma\operatorname{Tr}_{q}^{q^2}(\delta x)(\operatorname{Tr}_{q}^{q^2}(\beta x))^2+\gamma(\operatorname{Tr}_{q}^{q^2}(\delta x))^2\notag\\&+\gamma(\operatorname{Tr}_{q}^{q^2}(\delta x))^2\operatorname{Tr}_{q}^{q^2}(\beta x)+\gamma(\operatorname{Tr}_{q}^{q^2}(\delta x))^3. \label{5a}
	\end{align}
	Now assuming $f_5\circ U(x)=U(y)=\operatorname{Tr}_{q}^{q^2}(\beta y)+\gamma\operatorname{Tr}_{q}^{q^2}(\delta y)$, we get the following system of equations
	
	\begin{align}
		\operatorname{Tr}_q^{q^2}(\beta y)=&\operatorname{Tr}_q^{q^2}(\beta x)\label{5b}\\
		\operatorname{Tr}_q^{q^2}(\delta y)=&\operatorname{Tr}_q^{q^2}(\delta x)+\operatorname{Tr}_q^{q^2}(\delta x)(\operatorname{Tr}_q^{q^2}(\beta y))^2+(\operatorname{Tr}_q^{q^2}(\delta x))^2\notag\\&+(\operatorname{Tr}_q^{q^2}(\delta x))^2\operatorname{Tr}_q^{q^2}(\beta y)+(\operatorname{Tr}_q^{q^2}(\delta x))^3.\label{5c}
	\end{align}
	Equation \ref{5c} can be rewritten as
	\begin{align}
		(\operatorname{Tr}_q^{q^2}(\beta y)+\operatorname{Tr}_q^{q^2}(\delta x))^3+(\operatorname{Tr}_q^{q^2}(\delta x))^2+\operatorname{Tr}_q^{q^2}(\delta x)=\operatorname{Tr}_q^{q^2}(\delta y)+(\operatorname{Tr}_q^{q^2}(\beta y))^3.\notag
	\end{align}
	Taking $\operatorname{Tr}_q^{q^2}(\beta y)+\operatorname{Tr}_q^{q^2}(\delta x)=z$ in above equation, we get
	$$z^3+z^2+z=(z+1)^3+1=\operatorname{Tr}_q^{q^2}(\delta y)+(\operatorname{Tr}_q^{q^2}(\beta y))^3+(\operatorname{Tr}_q^{q^2}(\beta y))^2+\operatorname{Tr}_q^{q^2}(\beta y).$$
	Equivalently, we have
	 $$z=\large(\operatorname{Tr}_q^{q^2}(\delta y)+(\operatorname{Tr}_q^{q^2}(\beta y))^3+(\operatorname{Tr}_q^{q^2}(\beta y))^2+\operatorname{Tr}_q^{q^2}(\beta y)+1\large)^t+1.$$
	This implies that
	\begin{align}
		\operatorname{Tr}_q^{q^2}(\delta x)=&(\operatorname{Tr}_q^{q^2}(\delta y)+(\operatorname{Tr}_q^{q^2}(\beta y))^3+(\operatorname{Tr}_q^{q^2}(\beta y))^2+\operatorname{Tr}_q^{q^2}(\beta y)+1)^t\notag\\&+\operatorname{Tr}_q^{q^2}(\beta y)+1.\label{5d}
	\end{align}
	We write the equations \ref{5b} and \ref{5d} in matrix form 
	\begin{align}\label{5e}
		\begin{bmatrix}
		\beta&\beta^{q}\\
		\delta&\delta^{q}
	\end{bmatrix}
	\begin{bmatrix}
		x\\x^{q}
	\end{bmatrix}=
	\begin{bmatrix}
		\operatorname{Tr}_{q}^{q^2}(\beta y)\\
		N
	\end{bmatrix}\end{align}where,$$N=	\{\operatorname{Tr}_q^{q^2}(\delta y)+(\operatorname{Tr}_q^{q^2}(\beta y))^3+(\operatorname{Tr}_q^{q^2}(\beta y))^2+\operatorname{Tr}_q^{q^2}(\beta y)+1\}^t+\operatorname{Tr}_q^{q^2}(\beta y)+1.$$

 Applying elementary row operations in equation \ref{5e}, we can easily see that
	\begin{align}
		x=&(\beta\delta^q+\delta\beta^q)^{-1}[\delta^q	\operatorname{Tr}_{q}^{q^2}(\beta y)+\beta^q\{\operatorname{Tr}_q^{q^2}(\delta y)+(\operatorname{Tr}_q^{q^2}(\beta y))^3+(\operatorname{Tr}_q^{q^2}(\beta y))^2\notag\\&+\operatorname{Tr}_q^{q^2}(\beta y)+1\}^t+\beta^q\operatorname{Tr}_q^{q^2}(\beta y)+\beta^q].\label{5e}
	\end{align}
The rest of the proof is omitted.
\end{proof}

\begin{thm}Let $q=2^m$, $m$ an odd number, and $\gamma\in \mathbb{F}_{2^2}\setminus\mathbb{F}_2$. Then the compositional inverse of permutation polynomial 
	$f_6(x)=x+\gamma\operatorname{Tr}_{q}^{q^2}(x+x^{q+2})$ over $\mathbb{F}_{q^2}$ is
	$f_6^{-1}(x)=\operatorname{Tr}_q^{q^2}((1+\gamma)x)+\gamma[\operatorname{Tr}_q^{q^2}((1+\gamma)x)+\{\operatorname{Tr}_q^{q^2}(x)+(\operatorname{Tr}_q^{q^2}((1+\gamma)x))^3\}^t]$, where $3t\equiv 1\mod{(2^m-1)}$.
\end{thm}

\begin{proof}
We can easily see that
	\begin{align}
		f_6\circ U(x) 
		=& \operatorname{Tr}_q^{q^2}(\beta x)+\gamma[\operatorname{Tr}_q^{q^2}(\beta x)(\operatorname{Tr}_q^{q^2}(\delta x))^2+\operatorname{Tr}_q^{q^2}(\delta x)(\operatorname{Tr}_q^{q^2}(\beta x))^2\notag \\
		&+(\operatorname{Tr}_q^{q^2}(\delta x))^3].\label{6a}
	\end{align}
	From the equation $f\circ U(x)=U(y)=\operatorname{Tr}_{q}^{q^2}(\beta y)+\gamma\operatorname{Tr}_{q}^{q^2}(\delta y)$, we get the following system of equations
	\begin{align}
	\operatorname{Tr}_q^{q^2}(\beta y)=&\operatorname{Tr}_q^{q^2}(\beta x)\label{6b}\\
	\operatorname{Tr}_q^{q^2}(\delta y)=&\operatorname{Tr}_q^{q^2}(\beta y)(\operatorname{Tr}_q^{q^2}(\delta x))^2+(\operatorname{Tr}_q^{q^2}(\beta y))^2\operatorname{Tr}_q^{q^2}(\delta x)+(\operatorname{Tr}_q^{q^2}(\delta x))^3.\label{6c}
	\end{align}
 The equation \ref{6c} can be rewritten as follows 
	$$\operatorname{Tr}_q^{q^2}(\delta y)=(\operatorname{Tr}_q^{q^2}(\beta y)+\operatorname{Tr}_q^{q^2}(\delta x))^3+(\operatorname{Tr}_q^{q^2}(\beta y))^3.$$
	Or equivalently we have
	\begin{align}
		\operatorname{Tr}_q^{q^2}(\delta x)=&(\operatorname{Tr}_q^{q^2}(\delta y)+(\operatorname{Tr}_q^{q^2}(\beta y))^3)^t+\operatorname{Tr}_q^{q^2}(\beta y).\label{6d}
	\end{align}
	We write the equations \ref{6b} and \ref{6d} in matrix form 
	\begin{align}
		\begin{bmatrix}
		\beta&\beta^{q}\\
		\delta&\delta^{q}
	\end{bmatrix}
	\begin{bmatrix}
		x\\x^{q}
	\end{bmatrix}=
	\begin{bmatrix}
		\operatorname{Tr}_{q}^{q^2}(\beta y)\\
		(\operatorname{Tr}_q^{q^2}(\delta y)+(\operatorname{Tr}_q^{q^2}(\beta y))^3)^t+\operatorname{Tr}_q^{q^2}(\beta y)
	\end{bmatrix}.
\end{align}
On solving the above equation, we can easily get 
	\begin{align}
		x=&(\beta\delta^q+\delta\beta^q)^{-1}[\delta^q\operatorname{Tr}_q^{q^2}(\beta y)+\beta^q(\operatorname{Tr}_q^{q^2}(\delta y)+(\operatorname{Tr}_q^{q^2}(\beta y))^3)^t+\beta^q\operatorname{Tr}_q^{q^2}(\beta y)].\label{6e}
	\end{align}
The rest of the proof is similar and has been omitted.

\end{proof}

\end{document}